\documentclass[12pt,a4paper,oneside]{amsart}
\usepackage{a4,a4wide}
\usepackage{amsfonts,amsmath,amssymb,amsthm,enumitem}
\usepackage{mathtools}

\usepackage{color}
\usepackage{ifpdf}
\ifpdf
    \usepackage[pdftex]{graphicx}
    \DeclareGraphicsExtensions{.png,.pdf,.jpg}
\else
   \usepackage[dvips]{graphicx}
   \DeclareGraphicsExtensions{.eps}
\fi
\graphicspath{{.}{figures/}}
\usepackage[latin1]{inputenc}
\numberwithin{equation}{section}

\usepackage{doi}

\newtheorem{theorem}{Theorem}[section]
\newtheorem{lemma}[theorem]{Lemma}

\newtheorem{remark}[theorem]{Remark}
\newtheorem*{theorem*}{Theorem}
\newtheorem*{lemma*}{Lemma}
\newtheorem*{proposition*}{Proposition}
\newtheorem*{corollary*}{Corollary}
\def\R{\mathbb{R}}

\def\T{\mathbb{T}}

\def\Z{\mathbb{Z}}
\def\N{\mathbb{N}}
\def\EE{\mathbb{E}}

\def\P{\mathbb{P}}

\renewcommand{\phi}{\varphi}

\def\1{\mathbf{1}}

\def\XXint#1#2#3{{\setbox0=\hbox{$#1{#2#3}{\int}$ }
\vcenter{\hbox{$#2#3$ }}\kern-.57\wd0}}

\def\lt{\left}
\def\rt{\right}
\def\les{\lesssim}

\def\Id{\textrm{Id}}
\newcommand{\bra}[1]{\left( #1 \right)}
\newcommand{\sqa}[1]{\left[ #1 \right]}

\newcommand{\abs}[1]{\left| #1 \right|}
\newcommand{\nor}[1]{\left\| #1 \right\|}

\begin{document}
\title[Almost sharp rates in the optimal matching problem]{
Almost sharp rates of convergence for the average cost and displacement in the optimal matching problem}
\author{Michael Goldman \address[Michael Goldman]{CMAP, CNRS, \'Ecole polytechnique, Institut Polytechnique de Paris, 91120 Palaiseau,
France} \email{michael.goldman@cnrs.fr} Martin Huesmann \address[Martin Huesmann]{Universit\"at M\"unster,  Germany} \email{martin.huesmann@uni-muenster.de}
 %\and
  Felix Otto \address[Felix Otto]{Max Planck Institute for Mathematics in the Sciences, 04103 Leipzig, Germany} \email{Felix.Otto@mis.mpg.de}}

\begin{abstract}
In this note we prove estimates for the average cost in the quadratic optimal transport problem on the two-dimensional flat torus which are optimal up to a double logarithm. We also prove sharp estimates on the displacement. This is  based on the combination of a post-processing of our  quantitative linearization result together with a quasi-orthogonality property.
\end{abstract}

\date{\today}
\maketitle

\section{Introduction}
The aim of this note is to  improve the currently best known rates of convergence of the average cost in the quadratic optimal transport problem on the two-dimensional flat torus from \cite{AmGl19}. Compared with the conjecture from Caracciolo and al. in \cite{CaLuPaSi14}, our rate is optimal up to a double logarithm. We first use our
 quantitative linearization result from \cite{GHO} (in its post-processed version of \cite{GoHu22}) to improve the estimates from \cite{AmGlaTre} on the optimal transport map. We then combine this with a quasi-orthogonality property first observed in \cite[(3.28)]{GO} and  relatively standard heat kernel estimates to conclude.\\

To state our main result let us set some notation. We work on the 2 dimensional flat torus $\T_1=(\R/\Z)^2$  and  consider $(X_i)_{i\ge 1}$  a family of i.i.d. uniformly distributed random variables on $\T_1$. For $n\ge 1$ we  define the empirical measure as
$$ \mu_n=\frac1n \sum_{i=1}^n\delta_{X_i}.$$
The (quadratic) optimal matching problem on the torus is then 
$$ \EE\sqa{\inf_{\pi\in \mathsf{Cpl}(\mu_n, 1)} \int_{\T_1\times \T_1} |x-y|^2 d\pi(x,y) }=: \EE\sqa{W^2_2(\mu_n,1)}.$$
Here, we identify the Lebesgue measure on $\T_1$ with the constant density 1, $\mathsf{Cpl}(\mu_n, 1)$ denotes the couplings between $\mu_n$ and the Lebesgue measure on $\T_1$, and $W_2$ is the $L^2$ Kantorovich Wasserstein distance on $\T_1$.

While it is known since the seminal article \cite{AKT84} that $ n\EE\sqa{W^2_2(\mu_n,1)}\sim \log n$, it was recently conjectured in  \cite{CaLuPaSi14} that
\begin{equation}\label{conj_Caracciolo}
 \lim_{n\to \infty} \bra{n\EE\sqa{W^2_2(\mu_n,1)} -\frac{\log n}{4\pi}} \in \R.
\end{equation}
Our main result in this direction is the following theorem.
\begin{theorem}\label{theo:maincost}
We have\footnote{The notation $A\les B$, which we use in output statements, means that there exists a universal
constant $C>0$ such that $A\le C B$.}
\begin{align}\label{eq:ncost}
\abs{n\EE\sqa{W^2_2(\mu_n,1)} - \frac{\log n}{4\pi}} \les \log \log n.
\end{align}

\end{theorem}
In light of conjecture \eqref{conj_Caracciolo}, the error in  \eqref{eq:ncost} is optimal up to a double logarithm. This essentially improves by $\sqrt{\log n}$ the rate obtained in \cite{AmGl19} which built on the approach from \cite{AmStTr16} where the leading order term in \eqref{conj_Caracciolo} was identified. Let us however point out that our result currently holds only in the case of the flat torus while \cite{AmGl19,AmStTr16} covers any Riemannian manifold without boundary as well as the case of the unit cube. Moreover, since we mostly rely on \cite{GHO}, we are currently not able to treat the bi-partite problem.  We refer to \cite{ledoux2019optimal,Le17,BeCa,BaBo13, DeScSc13,goldman2021convergence,ambrosio2022quadratic,goldman2022optimal,caglioti2023random} for related  results.\\

As in \cite{AmStTr16,AmGl19} our proof of \eqref{eq:ncost} is based on the linearization ansatz proposed in this context by \cite{CaLuPaSi14} and which we now recall. If $\pi_n=(T_n, \Id)_\# 1$ is the optimal (random) coupling between $\mu_n$ and $1$, i.e.
\[
 W^2_2(\mu_n,1)=\int_{\T_1\times \T_1} |x-y|^2 d\pi_n,
\]
this ansatz postulate that $T_n(y)-y$ is well approximated by $\nabla f_n(y)$ where $f_n$ solves the Poisson equation
\[
 -\Delta f_n =\mu_n-1.
\]
As understood since \cite{AmStTr16}, this ansatz can only make sense after some regularization. For $t>0$, let  $p_t$ be the heat kernel at time $t$ on $\T_1$ and set $f_{n,t}=p_t\ast f_n$ be a  solution of
\[
 -\Delta f_{n,t} =p_t\ast(\mu_n-1).
\]
As a consequence of  the trace formula\footnote{when integrating with respect to the Lebesgue measure we drop the factor $dy$}, see e.g. \cite[Lemma 3.14]{AmGl19},
\begin{equation}\label{traceformula}
 n\EE\sqa{\int_{\T_1} |\nabla f_{n,t}|^2}= \frac{|\log t|}{4\pi} + O(\sqrt{t}).
\end{equation}
In order to prove \eqref{eq:ncost}, it is thus 'enough' to prove that for some $t_n$ with $|\log t_n|= \log n +O(\log \log n)$,
\[
 n\abs{\EE\sqa{W^2_2(\mu_n,1)} - \EE\sqa{\int_{\T_1} |\nabla f_{n,t}|^2} } \les \log \log n.
\]
The main step to prove this is our second main result.

\begin{theorem}\label{theo:mainmap}
For any $t\ge \frac1n=r_n^2$ we have
\begin{equation}\label{eq:mapPoisson}
n\EE\sqa{\int_{\T_1} \abs{x-y-\nabla f_{n,t}(x)}^2d\pi_n} \les 1 + \log\bra{\frac{t}{r_n^2}}.
\end{equation}
Moreover, for $t\ge t_n=   n^{-1} \log^3 n$ we have
\begin{align}\label{eq:nmap}
n\EE\sqa{\int_{\T_1} \abs{T_n(y)-y-\nabla f_{n,t}(y)}^2} \les \log\bra{\frac{t}{r_n^2}}.
\end{align}
\end{theorem}
Notice that the difference between \eqref{eq:mapPoisson} and \eqref{eq:nmap} is that we replaced $\nabla f_{n,t}(x)$ in the former  by $\nabla f_{n,t}(y)$ in the latter. While both have sharp dependence in $t$, we can go down to the microscopic scale $t=r_n^2=n^{-1}$ in \eqref{eq:mapPoisson} but are restricted to mesoscopic scales $t\ge t_n\gg n^{-1}$ in \eqref{eq:nmap}. This is most likely an artefact from our proof. Indeed, we derive \eqref{eq:nmap} from \eqref{eq:mapPoisson} combined with an $L^\infty$ bound on $\nabla^2 f_{n,t_n}$. This imposes the choice $t_n\gg n^{-1}$ (see \eqref{LinftyHess} of  Lemma \ref{lem:heat}). Still, \eqref{eq:nmap}
 improves  by $\sqrt{\log n}$ a similar bound from \cite{AmGlaTre} (see also the recent generalization \cite{ClMa23}).

  The proof of \eqref{eq:mapPoisson}  is mostly based on the quantitative linearization result from \cite{GHO} in its post-processed version from \cite{GoHu22}. Let us list the differences between \cite[Proposition 4.7]{GoHu22} and  \eqref{eq:mapPoisson}. A first point is to pass from a compactly supported convolution kernel as in \cite{GoHu22} to the heat kernel as in \eqref{eq:mapPoisson}. This is done using Lemma \ref{lem:changekernel} which relies on relatively standard heat/Green kernel estimates.  A second difficulty is  to pass from a quenched and localized estimate in \cite{GoHu22} to an annealed and global one, see \eqref{ncost:x2}. This is obtained appealing to stationarity. The argument here is a bit more delicate than its counterpart in \cite{GoHu22}.

 From this sketch of proof it is clear that the 'only' obstacle to obtain \eqref{eq:nmap} down to $t=n^{-1}$ is the fact that \cite{GHO} is currently only known for constant target measures. An extension of this result to arbitrary measures should also allow to extend our results to the bi-partite case.\\

Let us notice that combining \eqref{eq:nmap}, \eqref{traceformula} and $n\EE\sqa{W^2_2(\mu_n,1)}\sim \log n$ together with Cauchy-Schwarz inequality, it is not hard to obtain
\[
 \abs{n\EE\sqa{W^2_2(\mu_n,1)} - \frac{\log n}{4\pi}} \les (\log n \log\log n)^{\frac{1}{2}}
\]
which gives an alternative proof of the estimate in \cite{AmGl19}.
Similar sub-optimal error terms coming from the application of Cauchy-Schwarz inequality can be seen in \cite[Theorem 1.2]{AmGl19} for example. In order to obtain the sharper  estimate \eqref{eq:ncost} we rely instead on the quasi-orthogonality property
\[
 \abs{n\EE\sqa{\int_{\T_1} (T_n(y)-y-\nabla f_{n,t_n}(y))\cdot \nabla f_{n,t_n}(y)}}\les 1.
\]
This type of estimates, first noticed in  \cite[(3.28)]{GO},  see also \cite[Lemma 1.7]{GHO}, are a central ingredient in the variational approach to the regularity theory for optimal transport maps (see also \cite{KoOt,koch,miuOtt}).

\section{Preliminaries}
In this section we gather a few technical results which follow from relatively standard heat-kernel estimates. To simplify notation and presentation we provide estimates only for moments of order four but similar bounds can be obtained for moments of arbitrary order. As in \cite{AmGl19}, for $t\ge0$ we let
\[
 q_t(x)=\int_t^\infty (p_s(x)-1)ds
\]
so that
\begin{equation}\label{eq:frep}
 \nabla f_{n,t}(y)=\frac{1}{n}\sum_{i=1}^n \nabla q_t(X_i-y).
\end{equation}

\begin{lemma}\label{lem:heat}
Let  $t_n=  n^{-1} \log^3 n$.  Then,
\begin{equation}\label{LinftyHess}\EE[\nor{\nabla^2 f_{n,t_n}}_\infty^4]^{\frac{1}{4}}\les \frac1{\log n}.
\end{equation}
For every $n^{-1}\le s<t<1$ we have
\begin{equation}\label{heat:changetime}
 n\EE\lt[\int_{\T_1} |\nabla f_{n,s}-\nabla f_{n,t}|^4\rt]^{\frac{1}{2}}\les 1+\log\lt(\frac{t}{s}\rt).
\end{equation}

\end{lemma}
\begin{proof}
We first prove \eqref{LinftyHess}. For $\xi>0$ define the event $A_\xi^{n,t}=\{\nor{\nabla^2 f_{n,t}}_\infty\leq \xi\}.$ By \cite[Theorem 3.3]{AmGl19},
 there exists a constant $C>0$ such that for any $n\in \N$ and $0<t<1$ we have
\begin{equation}\label{thm:hessian}
\P[\bra{A_\xi^{n,t}}^c]\les \begin{cases} \frac{1}{t^2 \xi^3}e^{-Cnt \xi^2} & \text{ if } 0<\xi \le 1 \\
 \frac{1}{t^2 \xi^3}e^{-Cnt \xi} & \text{ if }  \xi \ge 1.
 \end{cases}
\end{equation}
The estimate for $\xi\ge 1$ is not explicitly contained in the statement of \cite[Theorem 3.3]{AmGl19} but follows by the exact same argument.

Since  for  a non-negative random variable $X$ and $a>0$, $\EE[X^p] \les a^p + \int_a^\infty \xi^{p-1}\P[X\geq \xi] d\xi$, \eqref{thm:hessian} implies that  for  every $1>a>0$,
\begin{equation}\label{toproveinftyhess}
\EE\sqa{\nor{\nabla^2 f_{n,t_n}}_\infty^4}\les a^4+ \int_a^1  \frac{1}{t_n^2} e^{-Cnt_n \xi^2} d\xi + \int_1^\infty  \frac1{t_n^2} e^{-Cnt_n\xi}d\xi.
\end{equation}
Recalling that $t_n= n^{-1} \log^3 n$, the last integral can be estimated as
\begin{equation*}
\frac1{t_n^2}\int_1^\infty   e^{-Cnt_n\xi}d\xi\les \frac{n^2}{\log^9 n } \exp(-C \log^3 n)
 \ll \frac1{\log^4 n}.
\end{equation*}
The other integral can be estimated  by 
\begin{multline*}
  \int_a^1  \frac{1}{t_n^2} e^{-Cnt_n \xi^2} d\xi \les \frac{n^2}{\log^6 n} \int_a^\infty \exp(- C \xi^2\log^3 n ) d\xi \\\stackrel{\xi=  s\log^{-3/2} n}{\les}
  \frac{n^2}{\log^{\frac{15}{2}} n}\int_{a \log^{3/2} n}^\infty e^{- C s^2} ds
 \les  \frac{n^2}{\log^9 n} e^{- C a^2 \log^{3 }n}.
\end{multline*}
Choosing $a=\gamma \log^{-1} n$ with $C\gamma^2>2$ we get
\[
  e^{- C a^2 \log^{3 }n}\le n^{-2}
\]
and thus
\[
 \int_a^1  \frac{1}{t_n^2} e^{-Cnt_n \xi^2} d\xi \les \frac{1}{\log^4 n}.
\]
Plugging this in \eqref{toproveinftyhess} concludes the proof of \eqref{LinftyHess}.\\

We now turn to the proof of \eqref{heat:changetime}. By stationarity we have
\[
 \EE\lt[\int_{\T_1} |\nabla f_{n,s}-\nabla f_{n,t}|^4\rt]=\EE[|\nabla f_{n,s}(0)-\nabla f_{n,t}(0)|^4].
\]
Since $\nabla f_{n,s}(0)= n^{-1}\sum_i \nabla q_s(X_i)$ we have by Rosenthal inequality,
\begin{multline*} \EE[|\nabla f_{n,s}(0)-\nabla f_{n,t}(0)|^4]\les n^{-4} \lt( n \int_{\T_1} |\nabla q_s-\nabla q_t|^4 + \lt(n \int_{\T_1} |\nabla q_s-\nabla q_t|^2 \rt)^2\rt)
\\
=n^{-3}\int_{\T_1} |\nabla q_s-\nabla q_t|^4 +n^{-2}\lt(\int_{\T_1} |\nabla q_s-\nabla q_t|^2 \rt)^2.
\end{multline*}
To estimate the first right-hand side term we recall from \cite[Corollary 3.13]{AmGl19} that for $0<t<1$,
\begin{equation}\label{fourthmomentq}
 \int_{\T_1} |\nabla q_t|^4\les t^{-1}.
\end{equation}
Using triangle inequality we thus conclude that
\begin{equation}\label{heat:1}
 \int_{\T_1} |\nabla q_s-\nabla q_t|^4\les \int_{\T_1} |\nabla q_s|^4+\int_{\T_1} |\nabla q_t|^4 \les  t^{-1}+ s^{-1}\les n.
\end{equation}
For the second right-hand side  term we can argue as in \cite[Proposition 3.11]{AmGl19} using that  $q_s=\int_s^\infty (p_r-1)dr$ and $-\Delta q_s=p_s-1$ to obtain after integration by parts,
\begin{multline*}
\int_{\T_1} |\nabla q_s-\nabla q_t|^2= \int_{\T_1} (q_t-q_s) (p_t-p_s)
= \int_{\T_1}\int_s^t p_r (p_t-p_s) dr dx \\
= \int_s^t\int_{\T_1} (p_r p_t-1) dx dr - \int_s^t\int_{\T_1} (p_r p_s-1) dx dr.
\end{multline*}
Using the semi-group property of the heat kernel together with the trace formula (see e.g. \cite[Theorem 3.7]{AmGl19}) we have
\begin{multline*}
 \int_s^t\int_{\T_1} (p_r p_t-1) dx dr =\int_s^t (p_{r+t}(0)-1) dr
 =\int_{s+t}^{2t} (p_{r}(0)-1) dr=\frac{1}{4\pi} \log\lt(\frac{2t}{s+t}\rt) +O(1).
\end{multline*}
Arguing similarly we have
\[
 \int_s^t\int_{\T_1} (p_r p_s-1) dx dr= \frac{1}{4\pi} \log\lt(\frac{s+t}{2s}\rt) +O(1)
\]
so that
\[
 \EE\sqa{\int_{\T_1} |\nabla q_s-\nabla q_t|^2}= \frac{1}{4\pi}\log\lt(\frac{t}{s}\rt) +O(1).
\]
Combining this with \eqref{heat:1} concludes the proof of \eqref{heat:changetime}.
\end{proof}

\begin{lemma}\label{lem:comp}
For $t\ge r^2_n= \frac1{n}$ we have
\begin{equation}\label{changezerox}
n \EE\sqa{ \frac{1}{|B_{r_n}|}\int_{B_{r_n}} \abs{\nabla f_{n,t}(0)-\nabla f_{n,t}(x) }^2d\mu_n } \les \frac{1}{(nt)^{1/2}} \les 1.
\end{equation}
\end{lemma}
\begin{proof}
Put $E=\EE\sqa{ \int_{B_{r_n}} \abs{\nabla f_{n,t}(0)-\nabla f_{n,t}(x) }^2d\mu_n }$ and write $\mu_n=\frac{1}{n}\sum_{i=1}^n \delta_{X_i}$ to get 
\begin{multline*}
 E=\EE\sqa{ \chi_{B_{r_n}}(X_1)|\nabla f_{n,t}(0)-\nabla f_{n,t}(X_1)|^2 }=\EE\sqa{ \chi_{B_{r_n}}(X_1)|\nabla f_{n,t}(0)|^2 }\\
 -2\EE\sqa{ \chi_{B_{r_n}}(X_1)\nabla f_{n,t}(0)\cdot\nabla f_{n,t}(X_1) }+\EE\sqa{ \chi_{B_{r_n}}(X_1)|\nabla f_{n,t}(X_1)|^2 }.
\end{multline*}
We now estimate each term separately using \eqref{eq:frep}. For the first one we have 
\[
 \EE\sqa{ \chi_{B_{r_n}}(X_1)|\nabla f_{n,t}(0)|^2 }=\frac{1}{n^2} \sum_{i,j} \EE\sqa{ \chi_{B_{r_n}}(X_1) \nabla q_t(X_i)\cdot \nabla q_t(X_j)  }.
\]
By independence of the $X_i$ and the fact that $\int_{\T} \nabla q_t(x)=0$ we obtain that for $i\neq j$ the expectation is $0$. Hence, 
\begin{multline*}
 \EE\sqa{ \chi_{B_{r_n}}(X_1)|\nabla f_{n,t}(0)|^2 }=\frac{n-1}{n^2} \EE\sqa{ \chi_{B_{r_n}}(X_1) |\nabla q_t(X_2)|^2  }+ \frac{1}{n^2}\EE\sqa{ \chi_{B_{r_n}}(X_1) |\nabla q_t(X_1)|^2  }\\
 =\frac{n-1}{n^2} |B_{r_n}|\int_{\T} |\nabla q_t|^2 +\frac{1}{n^2} \int_{B_{r_n}} |\nabla q_t|^2.
\end{multline*}
For the last term we have similarly,
\begin{multline*}
 \EE\sqa{ \chi_{B_{r_n}}(X_1)|\nabla f_{n,t}(X_1)|^2 }=\frac{1}{n^2}\sum_{i,j}\EE\sqa{ \chi_{B_{r_n}}(X_1)\nabla q_{t}(X_i-X_1)\cdot \nabla q_t(X_j-X_1) }\\
 =\frac{n-1}{n^2} \EE\sqa{ \chi_{B_{r_n}}(X_1) |\nabla q_t(X_2-X_1)|^2  }+\frac{1}{n^2} \EE\sqa{ \chi_{B_{r_n}}(X_1) |\nabla q_t(0)|^2  }\\
 =\frac{n-1}{n^2} |B_{r_n}|\int_{\T} |\nabla q_t|^2  +\frac{1}{n^2} |B_{r_n}|  |\nabla q_t(0)|^2.
\end{multline*}
Regarding the middle term we have 
\begin{multline*}
 \EE\sqa{ \chi_{B_{r_n}}(X_1)\nabla f_{n,t}(0)\cdot\nabla f_{n,t}(X_1) }=\sum_{i,j} \EE\sqa{ \chi_{B_{r_n}}(X_1)\nabla q_t(X_i)\cdot\nabla q_{t}(X_j-X_1) }\\
 =\frac{n-1}{n^2} \EE\sqa{ \chi_{B_{r_n}}(X_1) \nabla q_t(X_2)\cdot\nabla q_{t}(X_2-X_1)  }+\frac{1}{n^2} \EE\sqa{ \chi_{B_{r_n}}(X_1) \nabla q_t(X_1)\cdot\nabla q_{t}(0)  }\\
 =\frac{n-1}{n^2} \int_{B_{r_n}}\int_{\T} \nabla q_t(z)\cdot \nabla q_t(z-x)+\frac{1}{n^2}\int_{B_{r_n}}\nabla q_t(x)\cdot \nabla q_t(0).
\end{multline*}
Since  by \cite[Proposition 3.12]{AmGl19}
\begin{equation}\label{supqn}
 \sup_{B_{r_n}} |\nabla q_t|\les (r_n +t^{1/2})^{-1}\les t^{-1/2}
\end{equation}
we have 
\[
 \frac{1}{n^2}\int_{B_{r_n}}\nabla q_t(x)\cdot \nabla q_t(0)\les  \frac{|B_{r_n}|}{n^2} t^{-1}% \les \frac{|B_{r_n}|}{n}
\]
and similarly for the two other terms with prefactor $n^{-2}$. Therefore, for some $C\gg1$,
\begin{multline*}
  n E- C\frac{|B_{r_n}|}{n t}\les \lt(|B_{r_n}|\int_{\T} |\nabla q_t|^2-\int_{B_{r_n}}\int_{\T} \nabla q_t(z)\cdot \nabla q_t(z-x)\rt)\\
  =\int_{B_{r_n}}\int_{\T} \nabla q_t(z)\cdot (\nabla q_t(z)-\nabla q_t(z-x))\\
  =\int_{B_{r_n}}\int_{\T} -\Delta q_t(z)\cdot ( q_t(z)- q_t(z-x))\\
  =\int_{B_{r_n}}\int_{\T} (p_t(z)-1)\cdot ( q_t(z)- q_t(z-x))\\
  =\int_{B_{r_n}} (q_{2t}(0)-q_{2t}(x))\le |B_{r_n}| r_n \sup_{B_{r_n}}|\nabla q_{2t}|.
\end{multline*}
Here we used that $-\Delta q_t=p_t-1$ together with the semi-group property of $p_t$. Using \eqref{supqn} again  we conclude the proof of \eqref{changezerox}.
\end{proof}
\begin{remark}
 For $t\ge t_n=n^{-1}\log^3 n$, the proof of \eqref{changezerox} can be significantly simplified with the help of the Hessian bounds \eqref{LinftyHess}.
\end{remark}

Finally, in order to translate the results from \cite{GHO,GoHu22} to the setting of \cite{AmStTr16,AmGl19,AmGlaTre} we will need to be able to switch from convolutions against compactly supported kernels to convolutions against the heat kernel.
\begin{lemma}\label{lem:changekernel}
 Let $\eta\in C^\infty_c(B_1)$ be a smooth convolution kernel. For $r>0$, set  $\eta_r=r^{-2}\eta(\cdot/r)$ and let then $\phi_n^r$ the mean-zero solution of
 \[
  -\Delta  \phi_n^r=\eta_r\ast(\mu_n-1).
 \]
For every $t\ge n^{-1}$ we have
\begin{equation}\label{conclusionchangekernel}
 n \EE[|\nabla f_{n,t}(0)-\nabla \phi^{\sqrt{t}}_n(0)|^4]^{\frac{1}{2}}\les 1.
\end{equation}
\end{lemma}
\begin{proof}
 We start by noting that $\nabla \phi^{\sqrt{t}}_n(0)=n^{-1}\sum_i (\eta_{\sqrt{t}}\ast \nabla q_0)(X_i)$. Since  $\nabla f_{n,t}(0)=n^{-1}\sum_i \nabla q_t(X_i)$ we may apply as above Rosenthal inequality to obtain
\begin{multline*}
 \EE[|\nabla f_{n,t}(0)-\nabla \phi^{\sqrt{t}}_n(0)|^4]\\
 \les n^{-4}\lt(\sum_i \EE[|\nabla q_t(X_i)-\nabla (\eta_{\sqrt{t}}\ast q_0)(X_i)|^4] + (\sum_i \EE[|\nabla q_t(X_i)-\nabla (\eta_{\sqrt{t}}\ast q_0)(X_i)|^2])^2\rt)\\
 =n^{-4}\lt(n \int_{\T_1}|\nabla q_t-\nabla (\eta_{\sqrt{t}}\ast q_0)|^4 + (n \int_{\T_1}|\nabla q_t-\nabla (\eta_{\sqrt{t}}\ast q_0)|^2)^2\rt)\\
 \les n^{-3} \lt(\int_{\T_1}|\nabla q_t|^4+\int_{\T_1}|\nabla (\eta_{\sqrt{t}}\ast q_0)|^4\rt) + n^{-2} \lt(\int_{\T_1}|\nabla q_t-\nabla (\eta_{\sqrt{t}}\ast q_0)|^2\rt)^2.
\end{multline*}
The first right-hand side term is estimated by \eqref{fourthmomentq}. Arguing exactly as in the proof of \cite[(3.5)]{GoHu22}, we may similarly estimate the second right-hand side term by
\begin{equation*}\label{fourthmomenteta}
 \int_{\T_1} |\nabla (\eta_{\sqrt{t}}\ast q_0)|^4\les t^{-1}.
\end{equation*}
In order to conclude the proof of \eqref{conclusionchangekernel} we are left with proving
 \begin{equation}\label{mainestim}
  \int_{\T_1} |\nabla q_t- \nabla (\eta_{\sqrt{t}}\ast q_0)|^2 \les 1.
 \end{equation}
To this aim we write
\begin{equation}\label{mainestim:break}
 \int_{\T_1} |\nabla q_t- \nabla (\eta_{\sqrt{t}}\ast q_0)|^2 \les\int_{ B_{3 \sqrt{t}}} |\nabla q_t|^2 + \int_{B_{3\sqrt{t}}} |\nabla (\eta_{\sqrt{t}}\ast q_0)|^2 +\int_{\T_1\backslash B_{3\sqrt{t}}} |\nabla q_t- \nabla (\eta_{\sqrt{t}}\ast q_0)|^2
\end{equation}
For the first right-hand side term we notice that by \cite[Proposition 3.12]{AmGl19} in $B_{3\sqrt{t}}$ we have
\[
|\nabla q_t|\les \frac{1}{\sqrt{t}}
\]
so that
\begin{equation}\label{mainestim:1}
\int_{ B_{3 \sqrt{t}}} |\nabla q_t|^2\les \int_{ B_{3 \sqrt{t}}} \frac{1}{t}\les 1.
\end{equation}
The second right-hand side term in \eqref{mainestim:break} is treated similarly. Indeed, arguing as in \cite[(3.7)]{GoHu22} we get  that in $B_{3\sqrt{t}}$,
\[
 |\nabla (\eta_{\sqrt{t}}\ast q_0)|\les \frac{1}{\sqrt{t}}
\]
so that
\begin{equation}\label{mainestim:2}
 \int_{B_{3\sqrt{t}}} |\nabla (\eta_{\sqrt{t}}\ast q_0)|^2\les 1.
 \end{equation}
We finally turn to the last right-hand side term in \eqref{mainestim:break}. We further decompose it as
\begin{equation}\label{mainestim:last}
 \int_{\T_1\backslash B_{3\sqrt{t}}} |\nabla q_t- \nabla (\eta_{\sqrt{t}}\ast q_0)|^2\les \int_{\T_1\backslash B_{3\sqrt{t}}} |\nabla q_t- \nabla q_0|^2+\int_{\T_1\backslash B_{3\sqrt{t}}} | \nabla (\eta_{\sqrt{t}}\ast q_0)-\nabla q_0|^2.
\end{equation}
For the first right-hand side term we have by definition of $q_t$,
\[
 \nabla q_t- \nabla q_0=-\int_0^t \nabla p_s ds.
\]
Moreover, by \cite[Theorem 3.8\&3.9]{AmGl19}, we have in $\T_1\backslash B_{3\sqrt{t}}$
\[
 |\nabla p_s|^2\les \frac{|x|^2}{s^4} \exp(-c \frac{|x|^2}{s}).
\]
Using Cauchy-Schwarz, we find
\begin{multline*}
 \int_{\T_1\backslash B_{3\sqrt{t}}} |\nabla q_t- \nabla q_0|^2\les t\int_0^t \int_{\T_1\backslash B_{3\sqrt{t}}} |\nabla p_s|^2\les t\int_0^t \int_{\T_1\backslash B_{3\sqrt{t}}} \frac{|x|^2}{s^4} \exp(-c \frac{|x|^2}{s})\\
 \stackrel{x=\sqrt{s} y}{\les}t\int_0^t \int_{3\sqrt{t/s}}^\infty \frac{r^3}{s^2} \exp(- c r^2) dr ds.
\end{multline*}
Using that  for $A\ge 3$,
\[
 \int_A^\infty r^3 \exp(- c r^2) \les A^2 \exp(-c A^2)
\]
we find
\[
 \int_{\T_1\backslash B_{3\sqrt{t}}} |\nabla q_t- \nabla q_0|^2\les t^2\int_0^t s^{-3} \exp(-c \frac{t}{s}) ds\stackrel{s=t u}{=} \int_0^1 u^{-3}\exp(-\frac{1}{u}) du\les 1.
\]
We now estimate the second right-hand side term in \eqref{mainestim:last}. For $x\in \T_1\backslash B_{3\sqrt{t}}$ and $y\in B_{\sqrt{t}}$ we have arguing exactly as in \cite[Lemma 3.1]{GoHu22}
\[
 |\nabla q_0(x-y)-\nabla q_0(x)|\les \frac{|y|}{|x|^2}.
\]
Therefore,
\begin{multline*}
 \int_{\T_1\backslash B_{3\sqrt{t}}} \lt| \nabla (\eta_{\sqrt{t}}\ast q_0)-\nabla q_0\rt|^2=\int_{\T_1\backslash B_{3\sqrt{t}}} \lt|\int_{B_{\sqrt{t}}} \eta_{\sqrt{t}}(y) (\nabla q_0(x-y)-\nabla q_0(x))dy \rt|^2 dx\\
 \le \int_{\T_1\backslash B_{3\sqrt{t}}} \lt(\int_{B_{\sqrt{t}}} \eta_{\sqrt{t}}(y)  |y|\rt)^2 |x|^{-4} dx\le t \int_{\T_1\backslash B_{3\sqrt{t}}}  |x|^{-4} \les 1.
\end{multline*}
We conclude that
\begin{equation}\label{mainestim:3}
 \int_{\T_1\backslash B_{3\sqrt{t}}} |\nabla q_t- \nabla (\eta_{\sqrt{t}}\ast q_0)|^2\les 1.
\end{equation}
Injecting \eqref{mainestim:1}, \eqref{mainestim:2} and \eqref{mainestim:3} in \eqref{mainestim:break} we obtain the desired \eqref{mainestim}.
\end{proof}

We recall that from the concentration properties of $W^2_2(\mu_n,1)$ (see \cite{AmStTr16} or \cite[Theorem 4.5 \& Remark 4.6]{GoHu22}) we have the following moment bound:

\begin{lemma}\label{lem:Wmoments}
We have
\begin{equation}\label{eq:Wbound}
n\EE\sqa{(W^2_2(\mu_n,1))^2}^{\frac{1}{2}}\les \log n.
\end{equation}
\end{lemma}

\section{Proof of the main results}
In this section we prove Theorem \ref{theo:mainmap} and then Theorem \ref{theo:maincost}.

\begin{proof}[Proof of Theorem \ref{theo:mainmap}]
We start with \eqref{eq:mapPoisson}. Let $t\ge r_n^2$. As in \cite{GoHu22}, we first rely on  stationarity to infer that
\[
 \EE\sqa{\int_{\T_1\times \T_1} \abs{x-y-\nabla f_{n,t}(x)}^2d\pi_n } =\EE\sqa{ \frac{1}{|B_{r_n}|}\int_{B_{r_n}\times \T_1} \abs{x-y-\nabla f_{n,t_n}(x)}^2d\pi_n }.
\]
Let $\phi_n^r$ be defined as in Lemma \ref{lem:changekernel}. We now claim that in order to prove \eqref{eq:mapPoisson}, it is enough to prove
\begin{equation}\label{ncost:x2}
 n\EE\sqa{ \frac{1}{|B_{r_n}|}\int_{B_{r_n}\times \T_1} \abs{x-y-\nabla \phi_{n}^{r_n}(0)}^2d\pi_n }\les 1.
\end{equation}
To this aim we use for $(x,y)\in B_{r_n}\times \T_1$  the triangle inequality in the form
\begin{multline*}
 \abs{x-y-\nabla f_{n,t}(x)}^2\les \abs{x-y-\nabla \phi_n^{r_n}(0)}^2\\
 + \abs{\nabla \phi_n^{r_n}(0)-\nabla f_{n,r^2_n}(0)}^2+ \abs{\nabla f_{n,r^2_n}(0)-\nabla f_{n,t}(0)}^2+\abs{\nabla f_{n,t}(0)-\nabla f_{n,t}(x)}^2.
\end{multline*}
In order to prove the claim we need to show that the contributions coming from the last three terms on the right-hand side are controlled. Observe, that $n\mu_n(B_{r_n})$ is a Binomial random variable with all moments of order 1. Hence, we can estimate
\begin{multline*}
 n\EE\sqa{ \frac{1}{|B_{r_n}|}\int_{B_{r_n}\times \T_1} \abs{\nabla \phi_n^{r_n}(0)-\nabla f_{n,r^2_n}(0)}^2d\pi_n }\\
 =\frac{n}{|B_{r_n}|}\EE\sqa{ \mu_n(B_{r_n}) \abs{\nabla \phi_n^{r_n}(0)-\nabla f_{n,r^2_n}(0)}^2}\\
 \le \frac{n}{|B_{r_n}|}\EE\sqa{ (\mu_n(B_{r_n}))^2}^{\frac{1}{2}}\EE\sqa{ \abs{\nabla \phi_n^{r_n}(0)-\nabla f_{n,r^2_n}(0)}^4}^{\frac{1}{2}}\stackrel{\eqref{conclusionchangekernel}}{\les} 1.
\end{multline*}
Second,
\begin{multline*}
 n\EE\sqa{ \frac{1}{|B_{r_n}|}\int_{B_{r_n}\times \T_1} \abs{\nabla f_{n,r^2_n}(0)-\nabla f_{n,t}(0)}^2d\pi_n }\\
 \le \frac{n}{|B_{r_n}|}\EE\sqa{ (\mu_n(B_{r_n}))^2}^{\frac{1}{2}}\EE\sqa{ \abs{\nabla f_{n,r^2_n}(0)-\nabla f_{n,t}(0)}^4}^{\frac{1}{2}}\stackrel{\eqref{heat:changetime}}{\les}  \lt(1+\log\lt(\frac{t}{r^2_n}\rt)\rt).%\les \frac{\log \log n}{n}.
\end{multline*}
Finally, by \eqref{changezerox} of Lemma \ref{lem:comp} we have
\begin{multline*}
 n\EE\sqa{ \frac{1}{|B_{r_n}|}\int_{B_{r_n}\times \T_1} \abs{\nabla f_{n,t_n}(0)-\nabla f_{n,t_n}(x) }^2d\pi_n }\\=n\EE\sqa{ \frac{1}{|B_{r_n}|}\int_{B_{r_n}} \abs{\nabla f_{n,t_n}(0)-\nabla f_{n,t_n}(x) }^2d\mu_n }
 \les 1.
 % \frac{r_n^2}{|B_{r_n}|} \EE\sqa{ \mu_n(B_{r_n}) \nor{\nabla^2 f_{n,t_n}}_\infty^2}\\
 %\le \frac{r_n^2}{|B_{r_n}|} \EE\sqa{ \mu_n(B_{r_n})^2}^{\frac{1}{2}}\EE\sqa{ \nor{\nabla^2 f_{n,t_n}}_\infty^4}^{\frac{1}{2}}\stackrel{\eqref{LinftyHess}}{\les} \frac{r_n^2}{\log^2 n}\les \frac{1}{n\log^2 n},
\end{multline*}
We now establish \eqref{ncost:x2}. After rescaling and changing $\phi_n^{r_n}$ into $-\phi_n^{r_n}$, \cite[Proposition 4.6]{GoHu22} yields the existence of a random radius $r_{\ast, n}\ge r_n$  with moments of every order i.e. for every $p\ge 1$,
\begin{equation}\label{momentr}
 \EE\sqa{ \lt(\frac{r_{*,n}}{r_n}\rt)^p}\les_p 1
\end{equation}
such that
\[
 \sup\{|x-y- \nabla \phi_n^{r_{\ast,n}}(0)|\ : \ (x,y)\in \mathsf{Spt} \, \pi_n \cap (B_{r_n}\times \T_1)\}\les r_{\ast,n}.
\]
Moreover, combining \cite[Lemma  4.3 \& Theorem 4.5]{GoHu22} together with \cite[Lemma 2.10]{GHO} we may further assume that $r_{\ast,n}$ is such that for some $\alpha\in (0,1)$,
\begin{equation*}\label{delphi}
 |\nabla \phi_n^{r_{\ast,n}}(0)-\nabla \phi_n^{r_{n}}(0)|\les r_{\ast,n} \lt(\frac{r_{\ast,n}}{r_n}\rt)^{2+\alpha}.
\end{equation*}
We can thus estimate using triangle inequality,
\begin{multline*}
 n\EE\sqa{ \frac{1}{|B_{r_n}|}\int_{B_{r_n}\times \T_1} \abs{x-y-\nabla \phi_{n}^{r_n}(0)}^2d\pi_n }\\\les n \EE\sqa{ \frac{1}{|B_{r_n}|}\int_{B_{r_n}\times \T_1} \abs{x-y-\nabla \phi_{n}^{r_{\ast,n}}(0)}^2d\pi_n }+n \EE\sqa{ \frac{\mu(B_{r_n})}{|B_{r_n}|} \abs{\nabla \phi_{n}^{r_{\ast,n}}(0)-\nabla \phi_{n}^{r_n}(0)}^2 }\\
 \les n \EE\sqa{r_{*,n}^4}^{\frac{1}{2}}+ n \EE\sqa{r_{\ast,n}^4\lt(\frac{r_{\ast,n}}{r_n}\rt)^{4(2+\alpha)}}^{\frac{1}{2}}
 \stackrel{\eqref{momentr}}{\les} 1.
\end{multline*}
This concludes the proof of \eqref{ncost:x2} and, hence, of \eqref{eq:mapPoisson}.

To show \eqref{eq:nmap}
 we use the triangle inequality to estimate for $t\ge t_n$,
\begin{multline*}
n\EE\sqa{\int_{\T_1} \abs{T_n(y)-y-\nabla f_{n,t}(y)}^2}\\\les n\EE\sqa{\int_{\T_1} \abs{\nabla f_{n,t}-\nabla f_{n,t_n}}^2}+n\EE\sqa{\int_{\T_1} \abs{T_n(y)-y-\nabla f_{n,t_n}(y)}^2}.
\end{multline*}
Using H\"older inequality and \eqref{heat:changetime} we see that it is enough to prove \eqref{eq:nmap} for $t=t_n$. We then use triangle inequality again to write
\begin{multline*}
n\EE\sqa{\int_{\T_1} \abs{T_n(y)-y-\nabla f_{n,t_n}(y)}^2}=n\EE\sqa{\int_{\T_1\times \T_1} \abs{x-y-\nabla f_{n,t_n}(y)}^2 d\pi_n}\\
 \les n\EE\sqa{\int_{\T_1\times \T_1} \abs{x-y-\nabla f_{n,t_n}(x)}^2d\pi_n } + n\EE\sqa{\int_{\T_1\times \T_1} \abs{\nabla f_{n,t_n}(x)-\nabla f_{n,t_n}(y)}^2d\pi_n }.
\end{multline*}
Since the second right-hand side term is estimated as 
\begin{multline*}
 n\EE\sqa{\int_{\T_1\times \T_1} \abs{\nabla f_{n,t_n}(x)-\nabla f_{n,t_n}(y)}^2d\pi_n }\les  n\EE[\nor{\nabla^2 f_{n,t_n}}_\infty^4]^{\frac{1}{2}} \EE[(W^2_2(\mu_n,1))^2]^{\frac{1}{2}}\\
 \stackrel{\eqref{LinftyHess}\&\eqref{eq:Wbound}}{\les} \frac{\log n}{\log^2 n}=\frac1{\log n},
\end{multline*}
%in order to prove \eqref{eq:nmap} we are thus left with
%\begin{equation}\label{ncost:x}
% n\EE\sqa{\int_{\T_1\times \T_1} \abs{x-y-\nabla f_{n,t_n}(x)}^2d\pi_n } \les \log \log n.
%\end{equation}
 %Set $r_n= n^{-1/2}$. 
 we conclude by \eqref{eq:mapPoisson}.
 \end{proof}

% \begin{remark}\label{rem:optimal}
%  Notice that if we could prove the analog of \cite[Theorem 1.2]{GHO} (and as a consequence the analog of \cite[Proposition 4.7]{GoHu22}) but with the roles of $\mu_n$ and the Lebesgue measure exchanged, we would obtain the sharp estimate
%  \begin{equation}\label{rem:optimalmap}
%   n\EE\sqa{\int_{\T_1} \abs{T_n(y)-y-\nabla f_{n,t_n}(y)}^2 } \les 1
%  \end{equation}
%  with $t_n=n^{-1}$.
% \end{remark}
We finally prove Theorem \ref{theo:maincost}.
\begin{proof}[Proof of Theorem \ref{theo:maincost}]
  We start by writing
 \[
  |T_n(y)-y|^2= |T_n(y)-y-\nabla f_{n,t_n}(y)|^2 + |\nabla f_{n,t_n}(y)|^2 +2 (T_n(y)-y-\nabla f_{n,t_n}(y))\cdot \nabla f_{n,t_n}(y).
 \]
After integration and using \eqref{eq:nmap} with $t=t_n$, we thus get
\begin{multline*}
 \abs{n\EE\sqa{\int_{\T_1} |T_n(y)-y|^2} - \frac{\log n}{4\pi}} \les \log \log n +\abs{n\EE\sqa{\int_{\T_1} |\nabla f_{n,t_n}(y)|^2} - \frac{\log n}{4\pi}}\\
 + \abs{n\EE\sqa{\int_{\T_1} (T_n(y)-y-\nabla f_{n,t_n}(y))\cdot \nabla f_{n,t_n}(y)}}.
\end{multline*}
Since by the trace formula, see \eqref{traceformula}
\[
 \abs{n\EE\sqa{\int_{\T_1} |\nabla f_{n,t_n}(y)|^2} - \frac{\log n}{4\pi}}\les \log \log n,
\]
we are left with the proof of the quasi-orthogonality property
\begin{equation}\label{quasi-orth}
 \abs{n\EE\sqa{\int_{\T_1} (T_n(y)-y-\nabla f_{n,t_n}(y))\cdot \nabla f_{n,t_n}(y)}}\les 1.
\end{equation}
For this we  first split the left-hand side as
\begin{multline}\label{splitortho}
  n\EE\sqa{\int_{\T_1} (T_n(y)-y-\nabla f_{n,t_n}(y))\cdot \nabla f_{n,t_n}(y)}=n\EE\sqa{\int_{\T_1} (T_n(y)-y)\cdot \nabla f_{n,t_n}(y)}\\
 -n\EE\sqa{\int_{\T_1} | \nabla f_{n,t_n}(y)|^2}.
\end{multline}
To estimate the first term we argue in the spirit of \cite{KoOt} and  introduce the following notation: for $y\in \T_1$ let $X_s=X_s(y)$ be the constant speed geodesic with $X_0=y$ and $X_1=T_n(y)$. We thus have
\begin{multline*}
 \int_{\T_1} (T_n(y)-y)\cdot \nabla f_{n,t_n}(y)=\int_{\T_1}\int_0^1  \dot{X}_s\cdot \nabla f_{n,t_n}(X_0)\\
 =\int_{\T_1}\int_0^1  \dot{X}_s\cdot \nabla f_{n,t_n}(X_s)+\int_{\T_1}\int_0^1  \dot{X}_s\cdot (\nabla f_{n,t_n}(X_0)-\nabla f_{n,t_n}(X_s)).
\end{multline*}
For the second term we can estimate
\begin{multline}\label{eq:lastsuboptim}
 n\EE\sqa{\abs{\int_{\T_1}\int_0^1  \dot{X}_s\cdot (\nabla f_{n,t_n}(X_0)-\nabla f_{n,t_n}(X_s))}}\les n\EE\sqa{\nor{\nabla^2 f_{n,t_n}}_\infty W^2_2(\mu_n,1)}\\
 \stackrel{\eqref{LinftyHess}\&\eqref{eq:Wbound}}{\les} 1.
\end{multline}
For the first term we use that $\dot{X}_s\cdot \nabla f_{n,t_n}(X_s)= \frac{d}{ds} [f_{n,t_n}(X_s)]$ to write (recall that $X_1=T_n(y)$ and $X_0=y$)
\begin{multline*}
 \int_{\T_1}\int_0^1  \dot{X}_s\cdot \nabla f_{n,t_n}(X_s)=\int_{\T_1}(f_{n,t_n}(X_1)-f_{n,t_n}(X_0)) \\
 =\int_{\T_1} f_{n,t_n} d(\mu_n-1)= \int_{\T_1} f_{n,t_n} (-\Delta f_{n,0})=\int_{\T_1} \nabla f_{n,t_n}\cdot \nabla f_{n,0}.
 \end{multline*}
 Using the semi-group property of the heat kernel we get
 \[
  \int_{\T_1} \nabla f_{n,t_n}\cdot \nabla f_{n,0}=\int_{\T_1} |\nabla f_{n,\frac{t_n}{2}}|^2.
 \]
We thus conclude that
\[
 \abs{n\EE\sqa{\int_{\T_1} (T_n(y)-y)\cdot \nabla f_{n,t_n}(y)}- n\EE\sqa{\int_{\T_1} |\nabla f_{n,\frac{t_n}{2}}|^2}}\les 1.
\]
Plugging this back into \eqref{splitortho} yields
\begin{multline*}
 \abs{n\EE\sqa{\int_{\T_1} (T_n(y)-y-\nabla f_{n,t_n}(y))\cdot \nabla f_{n,t_n}(y)}}\\
 \les 1+ n \abs{\EE\sqa{\int_{\T_1} |\nabla f_{n,\frac{t_n}{2}}|^2}-\EE\sqa{\int_{\T_1} |\nabla f_{n,t_n}|^2}}\les 1
\end{multline*}
where in the last line we used once more \eqref{traceformula}.
This concludes the proof of \eqref{quasi-orth}.
\end{proof}
\begin{remark}\label{rem:optimalcost}
Let us point out that the only source of suboptimality in \eqref{eq:ncost} lies in \eqref{eq:lastsuboptim} where we use in a crucial way the bound on $\nabla^2 f_{n,t_n}$ from Lemma \ref{lem:heat}. Indeed, if we knew (using the notation from the proof of Theorem \ref{theo:maincost}) that analogously to \eqref{eq:mapPoisson}
\[
 n\EE\sqa{\int_{\T_1\times \T_1}\int_0^1 \abs{\dot{X}_s-\nabla f_{n,r_n^2}(X_s)}^2 dsd\pi_n }\les 1
\]
and analogously to \eqref{traceformula},
\begin{equation}\label{tracemodif}
 n\EE\sqa{\int_{\T_1\times \T_1}\int_0^1 |\nabla f_{n,t}(X_s)|^2 ds d\pi_n}= \frac{|\log t|}{4\pi} + O(\sqrt{t})
\end{equation}
then the exact same proof would yield an error term which is of order one.

% leading to the suboptimal  $\log\log n$following the reasoning of this proof either using the optimal estimate \eqref{eq:mapPoisson} with $t_n =n^{-1}$ or the conjectured estimate  \eqref{rem:optimalmap} for $t_n =n^{-1}$  we would still need to choose $t_n\gg n^{-1}$ (leading to the sub-optimal error term in $\log \log n$) is in \eqref{eq:lastsuboptim} (but nowhere else).
% Let us point out that if we could establish the optimal estimate \eqref{rem:optimalmap} for $t_n =n^{-1}$, the only place in the proof where we would still need to choose $t_n\gg n^{-1}$ (leading to the sub-optimal error term in $\log \log n$) is in \eqref{eq:lastsuboptim}.
\end{remark}
\begin{remark}
 In a similar direction, let us also observe that arguing  as in the proof of Lemma \ref{lem:comp} we can obtain the analog of \eqref{traceformula} (this explicit computation cannot be done for \eqref{tracemodif})
 \[
  n\EE\sqa{\int_{\T_1} |\nabla f_{n,t}|^2 d\mu_n}= \frac{|\log t|}{4\pi} + O(\sqrt{t}).
 \]
Instead of relying on \eqref{eq:nmap} in the proof of \eqref{eq:ncost}, we could have thus used  \eqref{eq:mapPoisson} (with $t=t_n$). This would however make the proof slightly heavier without affecting the final result.
\end{remark}

 \bibliographystyle{abbrv}

\bibliography{OT.bib}

\begin{thebibliography}{10}

\bibitem{AKT84}
M.~{Ajtai}, J.~{Koml\'os}, and G.~{Tusn\'ady}.
\newblock {On optimal matchings.}
\newblock {\em {Combinatorica}}, 4:259--264, 1984.

\bibitem{AmGl19}
L.~Ambrosio and F.~Glaudo.
\newblock Finer estimates on the {{\(2\)}}-dimensional matching problem.
\newblock {\em J. {\'E}c. Polytech., Math.}, 6:737--765, 2019.

\bibitem{AmGlaTre}
L.~Ambrosio, F.~Glaudo, and D.~Trevisan.
\newblock On the optimal map in the 2-dimensional random matching problem.
\newblock {\em Discrete Contin. Dyn. Syst.}, 39(12):7291--7308, 2019.

\bibitem{ambrosio2022quadratic}
L.~Ambrosio, M.~Goldman, and D.~Trevisan.
\newblock On the quadratic random matching problem in two-dimensional domains.
\newblock {\em Electronic Journal of Probability}, 27:1--35, 2022.

\bibitem{AmStTr16}
L.~Ambrosio, F.~Stra, and D.~Trevisan.
\newblock A {PDE} approach to a 2-dimensional matching problem.
\newblock {\em {Probab. Theory Relat. Fields}}, 173(1-2):433--477, 2019.

\bibitem{BaBo13}
F.~{Barthe} and C.~{Bordenave}.
\newblock {Combinatorial optimization over two random point sets.}
\newblock In {\em {S\'eminaire de probabilit\'es XLV}}, pages 483--535. Cham:
  Springer, 2013.

\bibitem{BeCa}
D.~Benedetto and E.~Caglioti.
\newblock Euclidean random matching in 2d for non-constant densities.
\newblock {\em Journal of Statistical Physics}, 181(3):854--869, 2020.

\bibitem{caglioti2023random}
E.~Caglioti and F.~Pieroni.
\newblock Random matching in 2d with exponent 2 for densities defined on
  unbounded sets.
\newblock {\em arXiv preprint arXiv:2302.02602}, 2023.

\bibitem{CaLuPaSi14}
S.~Caracciolo, C.~Lucibello, G.~Parisi, and G.~Sicuro.
\newblock Scaling hypothesis for the {E}uclidean bipartite matching problem.
\newblock {\em Physical Review E}, 90(1), 2014.

\bibitem{ClMa23}
N.~Clozeau and F.~Mattesini.
\newblock Annealed quantitative estimates for the quadratic 2d-discrete random
  matching problem.
\newblock {\em arXiv preprint arXiv:2303.00353}, 2023.

\bibitem{DeScSc13}
S.~Dereich, M.~Scheutzow, and R.~Schottstedt.
\newblock Constructive quantization: approximation by empirical measures.
\newblock {\em Ann. Inst. Henri Poincar\'e Probab. Stat.}, 49(4):1183--1203,
  2013.

\bibitem{GoHu22}
M.~Goldman and M.~Huesmann.
\newblock A fluctuation result for the displacement in the optimal matching
  problem.
\newblock {\em Ann. Probab.}, 50(4):1446--1477, 2022.

\bibitem{GHO}
M.~{Goldman}, M.~{Huesmann}, and F.~{Otto}.
\newblock Quantitative linearization results for the {M}onge-{A}mp\`ere
  equation.
\newblock {\em Com. Pure and Appl. Math}, 2021.

\bibitem{GO}
M.~Goldman and F.~Otto.
\newblock A variational proof of partial regularity for optimal transportation
  maps.
\newblock {\em Annales scientifiques de l'Ecole Normale Sup{\'e}rieure},
  53(5):1209--1233, 2020.

\bibitem{goldman2021convergence}
M.~Goldman and D.~Trevisan.
\newblock Convergence of asymptotic costs for random euclidean matching
  problems.
\newblock {\em Probability and Mathematical Physics}, 2(2):341--362, 2021.

\bibitem{goldman2022optimal}
M.~Goldman and D.~Trevisan.
\newblock Optimal transport methods for combinatorial optimization over two
  random point sets.
\newblock {\em Probability Theory and Related Fields}, pages 1--70, 2023.

\bibitem{koch}
L.~Koch.
\newblock Geometric regularisation for optimal transport with strongly p-convex
  cost.
\newblock {\em arXiv preprint arXiv:2303.10760}, 2023.

\bibitem{KoOt}
L.~Koch and F.~Otto.
\newblock Lecture notes on the harmonic approximation to quadratic optimal
  transport.
\newblock {\em arXiv preprint arXiv:2303.14462}, 2023.

\bibitem{Le17}
M.~Ledoux.
\newblock On optimal matching of {G}aussian samples.
\newblock {\em Zap. Nauchn. Sem. S.-Peterburg. Otdel. Mat. Inst. Steklov.
  (POMI)}, 457(Veroyatnost' \ i Statistika. 25):226--264, 2017.

\bibitem{ledoux2019optimal}
M.~Ledoux and J.-X. Zhu.
\newblock On optimal matching of gaussian samples iii.
\newblock {\em Probability and Mathematical Statistics}, 41, 2021.

\bibitem{miuOtt}
T.~Miura and F.~Otto.
\newblock Sharp boundary $\varepsilon$-regularity of optimal transport maps.
\newblock {\em Advances in Mathematics}, 381:107603, 2021.

\end{thebibliography}
\end{document}